\documentclass[a4paper, 10pt]{article}
\usepackage{mathrsfs,amsmath,amsthm}
\usepackage{graphicx, color, subfigure}
\usepackage{array}
\usepackage{cases}
\makeatletter
\def\th@plain{%
  \upshape 
}
\makeatother

\makeatletter
\renewenvironment{proof}[1][\proofname]{\par
  \pushQED{\qed}%
  \normalfont \topsep6\p@\@plus6\p@\relax
  \trivlist
  \item[\hskip\labelsep
        \bfseries
    #1\@addpunct{.}]\ignorespaces
}{%
  \popQED\endtrivlist\@endpefalse
}
\makeatother

\newtheorem{theorem}{Theorem}
\numberwithin{theorem}{section}

\newtheorem{corollary}{Corollary}

\newtheorem*{conjecture*}{Conjecture}

\newtheorem{claim}{Claim}

\newtheorem{observation}{Observation}

\theoremstyle{definition}

\usepackage[
  textheight=562pt, 
  textwidth=384pt, 
  centering,  
  headheight=50pt,
  headsep=12pt,
  footskip=12pt,
  footnotesep=24pt plus 2pt minus 12pt]{geometry}
\usepackage[T1]{fontenc}
\usepackage[varg]{txfonts}

\usepackage{enumerate}
\usepackage[square, numbers, sort&compress]{natbib}

\ifx\pdfoutput\undefined
 \usepackage[dvipdfm,%
  bookmarks=true,%
  bookmarksnumbered=true, 
  bookmarksopen=true, 
  plainpages=false,%
  pdfpagelabels,%
  colorlinks=true, 
  linkcolor=black, 
  citecolor=black,%
  hyperindex=true,
  urlcolor=black,
  pdfborder=001]{hyperref}
\else
 \usepackage[pdftex,%
  bookmarks=true,%
  bookmarksnumbered=true, 
  bookmarksopen=true, 
  plainpages=false,%
  pdfpagelabels,%
  colorlinks=true, 
  linkcolor=black, 
  citecolor=black,%
  anchorcolor=green,
  urlcolor= blue,
  breaklinks=true，
  hyperindex=true,
  pdfborder=001]{hyperref}
\fi
\hypersetup{%
  pdfstartview=FitH, 
  pdfauthor={Tao Wang}}

\newcommand{\etal}{et~al.\ }
\newcommand{\ie}{i.e.,\ }

\def\int(#1){\mathrm{int}(#1)}
\def\ext(#1){\mathrm{ext}(#1)}
\def\Int(#1){\mathrm{Int}(#1)}
\def\Ext(#1){\mathrm{Ext}(#1)}
\newcommand{\Lfloor}{\left\lfloor}
\newcommand{\Rfloor}{\right\rfloor}
\newcommand{\Lceil}{\left\lceil}
\newcommand{\Rceil}{\right\rceil}

\DeclareMathOperator {\diam}{diam}

\graphicspath{{figures/}}
\numberwithin{figure}{section}
\numberwithin{equation}{section}
\newtheorem*{Murty-Simon}{Murty-Simon Conjecture}
\newcommand{\un}{\mathrm{un}}
\begin{document}%
\title{On Murty-Simon Conjecture}
\author{Tao Wang\footnote{{\tt Corresponding
author: wangtao@henu.edu.cn}}\\
{\small \textsuperscript{a}\unskip Institute of Applied Mathematics}\\
{\small College of Mathematics and Information Science}\\
{\small Henan University, Kaifeng, 475004, P. R. China}}
\date{}
\maketitle

\begin{abstract}%
A graph is diameter two edge-critical if its diameter is two and the deletion of any edge increases the diameter. Murty and Simon conjectured that the number of edges in a diameter two edge-critical graph on $n$ vertices is at most $\left \lfloor \frac{n^{2}}{4} \right \rfloor$ and the extremal graph is the complete bipartite graph $K_{\left\lfloor
\frac{n}{2} \right\rfloor, \left\lceil \frac{n}{2} \right\rceil}$. In the series papers \cite{Haynes2011a, Haynes2011, Haynes2012}, the Murty-Simon Conjecture stated by Haynes \etal is
not the original conjecture, indeed, it is only for the diameter two edge-critical graphs of even
order. Haynes \etal proved the conjecture for the graphs whose complements have diameter three but only with even vertices. In this paper, we prove the Murty-Simon Conjecture for the graphs whose complements have diameter three, not only with even vertices but also odd ones.
\end{abstract}

\section{Introduction}
All graphs considered in this paper are simple. We adopt notation and terminology commonly used in the literature. Let $G$ be a graph with vertex set $V(G)$ and edge set $E(G)$. The
{\em neighborhood} of a vertex $v$ in a graph $G$, denoted by
$N_{G}(v)$, is the set of all the vertices adjacent to the vertex
$v$, \ie $N_{G}(v) = \{u \in V(G) \mid uv \in E(G)\}$, and the {\em
closed neighborhood} of a vertex $v$ in $G$, denoted by $N_{G}[v]$,
is defined by $N_{G}[v] = N_{G}(v) \cup \{v\}$. For a subset $S
\subseteq V$, the {\em neighborhood of the set $S$} in $G$ is the
set of all vertices adjacent to vertices in $S$, this set is denoted
by $N_{G}(S)$, and the {\em closed neighborhood of $S$} by $N_{G}[S]
= N_{G}(S) \cup S$. Let $S$ and $T$ be two subsets (not necessarily
disjoint) of $V(G)$, $[S, T]$ denotes the set of edges of $G$
with one end in $S$ and the other in $T$, and $e_{G}(S, T)=|[S,
T]|$. If every vertex in $S$ is adjacent to each vertex in $T$, then we say that $[S, T]$ is full. If $S \subseteq V(G)$, and $u, v$ are two nonadjacent vertices in $G$, then we say that $xy$ is a {\em missing edge} in $S$ (rather than ``$uv$ is a missing edge in $G[S]$'').

The {\em complement} $G^{c}$ of a simple graph $G = (V, E)$ is the
simple graph with vertex set $V$, two vertices are adjacent in
$G^{c}$ if and only if they are not adjacent in $G$.

Given a graph $G$ and two vertices $u$ and $v$ in it, the {\em distance}
between $u$ and $v$ in $G$, denoted by $d_{G}(u, v)$, is the length
of a shortest $u$-$v$ path in $G$; if there is no path connecting
$u$ and $v$, we define $d_{G}(u, v) = \infty$. The {\em diameter} of
a graph $G$, denoted by $\diam(G)$, is the maximum distance between any two vertices of $G$. Clearly, $\diam(G) = \infty$ if and only if $G$ is disconnected.

A subset $S \subseteq V$ is called a {\em dominating set} ({\bf DS})
of a graph $G$ if every vertex $v \in V$ is an element of $S$ or is
adjacent to a vertex in $S$, that is, $N_{G}[S] = V$. The {\em
domination number} of $G$, denoted by $\gamma(G)$, is the minimum
cardinality of a dominating set in $G$.

A subset $S \subseteq V$ is a {\em total dominating set},
abbreviated {\bf TDS}, of $G$ if every vertex in $V$ is adjacent to
a vertex in $S$, that is $N_{G}(S) = V$. Every graph without
isolated vertices has a TDS, since $V$ is a trivial TDS. The
{\em total domination number} of a graph $G$, denoted by
$\gamma_{t}(G)$, is the minimum cardinality of a TDS in $G$. For the graph with isolated vertices, we define its total domination number to be $\infty$. Total
domination in graphs was introduced by Cockayne, Dawes, and Hedetniemi
\cite{MR0584887}.

For two vertex subsets $X$ and $Y$, we say that {\em $X$ dominates $Y$} ({\em totally dominates $Y$}, respectively) if $Y \subseteq N_{G}[X]$ ($Y \subseteq N_{G}(X)$, respectively); sometimes, we also say that {\em $Y$ is dominated by $X$} ({\em totally dominated by $X$}, respectively).

For three vertices $u, v, w\in V(G)$, the symbol $uv\rightarrow w$ means that $\{u, v\}$ dominates $G - w$,
but $uw \notin E(G)$, $vw \notin E(G)$ and $uv \in E(G)$.

A graph $G$ is said to be {\em diameter-$d$ edge-critical} if $\diam(G) = d$
and $\diam(G - e) > \diam(G)$ for any edge $e \in E(G)$. Gliviak \cite{Gliviak1975} proved the impossibility of characterization of diameter-$d$ edge-critical graphs by finite extension or by forbidden subgraphs. Plesn\'{i}k
\cite{Plesn'ik1975} observed that all known minimal graphs of
diameter two on $n$ vertices have no more than $\left \lfloor
\frac{n^{2}}{4} \right \rfloor $ edges. Independently, Murty and
Simon (see \cite{Caccetta1979}) conjectured the following:
\begin{Murty-Simon}\label{MurtySimonConjecture}
If $G$ is a diameter-$2$ edge-critical graph on $n$ vertices, then $|E(G)| \leq \left
\lfloor \frac{n^{2}}{4} \right \rfloor$. Moreover, equality holds if
and only if $G$ is the complete bipartite graph $K_{\left\lfloor
\frac{n}{2} \right\rfloor, \left\lceil \frac{n}{2} \right\rceil}$.
\end{Murty-Simon}

Let $G$ be a diameter-$2$ edge-critical graph on $n$ vertices. Plesn\'{i}k
\cite{Plesn'ik1975} proved that $|E(G)| < 3n(n-1)/8$. Caccetta and H\"{a}ggkvist \cite{Caccetta1979} obtained that $|E(G)| < 0.27n^{2}$. Fan \cite{Fan1987} proved  the first part of the Murty-Simon Conjecture for $n \leq 24$ and for $n=26$; and 
\[
|E(G)| < \frac{1}{4}n^{2} + (n^{2}-16.2n + 56)/320 < 0.2532 n^{2}\]
for $n \geq 25$.
F\"{u}redi \cite{Furedi1992} proved the Murty-Simon Conjecture for $n > n_{0}$, where $n_{0}$ is not larger than a tower of $2$'s of height about $10^{14}$.

A graph is {\em total domination edge critical} if the addition of any edge decrease the total domination number. If $G$ is total domination edge critical with $\gamma_{t}(G) = k$, then we say that $G$ is a {\em $k$-$\gamma_{t}$-edge critical graph}. Haynes \etal \cite{Haynes1998} proved that the addition of an edge to a graph without isolated vertices can decrease the total domination number by at most two. A graph $G$ with the property that $\gamma_{t}(G) = k$ and $\gamma_{t}(G + e) = k-2$ for every missing edge $e$ in $G$ is called a {\em $k$-supercritical} graph.

\begin{theorem}[Hanson and Wang \cite{Hanson2003}]\label{DominationEdgeVsDiameterTwo}
A nontrivial graph $G$ is dominated by two adjacent vertices if and only if the
diameter of $G^{c}$ is greater than two.
\end{theorem}

\begin{corollary}%
A graph $G$ is diameter-$2$ edge-critical on $n$ vertices if and only if the total domination number of $G^{c}$ is greater than two but the addition of any edge in $G^{c}$ decrease the total domination number to be two, that is, $G^{c}$ is $K_{1} \cup K_{n-1}$ or $3$-$\gamma_{t}$-edge critical or $4$-supercritical.
\end{corollary}

The complement of $G$ is $K_{1} \cup K_{n-1}$ if and only if $G$ is $K_{1, n-1}$. Clearly, the Murty-Simon Conjecture holds for $K_{1, n-1}$.

The $4$-supercritical graphs are characterized in \cite{Merwe1998}.

\begin{theorem}
A graph $H$ is $4$-supercritical if and only if $H$ is the disjoint union of two nontrivial complete graphs. 
\end{theorem}

The complement of a $4$-supercritical graph is a complete bipartite graph. The Murty-Simon Conjecture holds for the graphs whose complements are $4$-supercritical, \ie complete bipartite graphs.

Therefore, we only have to consider the graphs whose complements are $3$-$\gamma_{t}$-edge critical. 

For $3$-$\gamma_{t}$-edge critical graphs, the bound on the diameter is established in \cite{Haynes1998}.
\begin{theorem}
If $G$ is a $3$-$\gamma_{t}$-edge critical graph, then $2 \leq \diam(G) \leq 3$.
\end{theorem}

Hanson and Wang \cite{Hanson2003} proved the first part of the Murty-Simon Conjecture for the graphs whose complements have diameter three. Recently, Haynes, Henning, van der Merwe and Yeo \cite{Haynes2011a} proved the second part for the graphs whose complements are $3$-$\gamma_{t}$-edge critical graphs with diameter three but only with even vertices. Also, Haynes \etal \cite{Haynes2012} proved the Murty-Simon Conjecture for the graphs of even order whose complements have vertex connectivity $\ell$, where $\ell = 1, 2, 3$. Haynes, Henning and Yeo \cite{Haynes2011} proved the Murty-Simon Conjecture for the graphs whose complements are claw-free.

In this paper, we prove the Murty-Simon Conjecture for the graphs whose complements are $3$-$\gamma_{t}$-edge critical graphs with diameter three, not only with even vertices but also odd ones. This theorem includes the result obtained by Haynes \etal \cite{Haynes2011}. We use the technique developed in \cite{Haynes2011}, and the proof is processed by a series of claims, a few claims are the same with them in \cite{Haynes2011}, but to make the paper self contained, we give a full proof of them.

Let $G$ be a $3$-$\gamma_{t}$-edge critical graph. Then the addition of any edge $e$ decrease the total domination number to be two, that is, $G + e$ is dominated by two adjacent vertices $x$ and $y$; we call such edge $xy$ {\em quasi-edge} of $e$. Note that $xy$ must contain at least one end of $e$. Clearly, quasi-edge of $e$ may not be unique. If $xy \mapsto w$, then $xy$ is quasi-edge of the missing edge $xw$, and also quasi-edge of missing edge $yw$; conversely, if $xy$ is quasi-edge of a missing edge, then there exists an unique vertex $w$ such that $xy \mapsto w$. So, if $xy \rightarrow w$, we write $\un(xy) = w$.

From the definition of $3$-$\gamma_{t}$-edge critical graph, we have the following frequently used observation.
\begin{observation}%
If $G$ is a $3$-$\gamma_{t}$-edge critical graph and $uv$ is a missing edge in it, then either
\begin{enumerate}[(i)]
\item $\{u, v\}$ dominates $G$; or
\item there exists a vertex $z$ such that $uz \mapsto v$ or $zv \mapsto u$.
\end{enumerate}
\end{observation}

For notation and terminology not defined here, we refer the reader to \cite{Haynes2011a}.
\section{Main results}
\begin{theorem}%
If $G$ is a $3$-$\gamma_{t}$-edge critical graph on $n$ vertices with diameter three , then $|E(G^{c})| < \Lfloor\frac{n^{2}}{4}\Rfloor$.
\end{theorem}
\begin{proof}%
Suppose, to the contrary, that $|E(G^{c})| \geq \Lfloor\frac{n^{2}}{4}\Rfloor$. Assume that $d_{G}(u_{0}, v_{0}) = 3$ and $\deg_{G}(u_{0}) \leq \deg_{G}(v_{0})$. Let $A = \{v \mid d_{G}(u_{0}, v) = 1\}$, $B = \{v \mid d_{G}(u_{0}, v) =2\}$, $C = \{v \mid d_{G}(u_{0}, v) =3\}$. Hence, $\{\{u_{0}\}, A, B, C\}$ is a partition of $V(G)$.
\begin{claim}\label{CClique}%
For every missing edge $e$ in $A$ or $B \cup C$, quasi-edges of $e$ are in $[A, B]$. Consequently, $C$ is a clique. Moreover, for every edge in $[A, B]$, it is quasi-edge of at most one missing edge in $A$ or $B \cup C$.
\end{claim}
\begin{proof}%
Suppose that $xy$ is a missing edge in $A$. Consider $G + xy$, since $\{x, y\}$ does not dominate $\{v_{0}\}$, there exists a vertex $z$ such that $xz \mapsto y$ or $zy \mapsto x$. In either case, neither $x$ nor $y$ dominate $v_{0}$, so $z$ dominates $v_{0}$, then $z \in N_{G}[v_{0}]$, and thus $z \in B$.

Suppose that $xy$ is a missing edge in $B \cup C$. Consider $G + xy$, since $\{x, y\}$ does not  dominate $\{u_{0}\}$, there exists a vertex $z$ such that $xz \mapsto y$ or $zy \mapsto x$. In either case, neither $x$ nor $y$ dominate $u_{0}$, then $z$ dominates $u_{0}$, and thus $z \in N_{G}(u_{0})=A$.

Let $uv$ be an arbitrary edge in $[A, B]$, by Observation 1, it is quasi-edge of at most one missing edge in $A$ or $B \cup C$.
\end{proof}

Now, we have
\begin{equation}\label{EQEQ}%
\Lfloor\frac{n^{2}}{4}\Rfloor \leq |E(G^{c})| \leq |A \cup \{u_{0}\}| \times |B \cup C| \leq \Lfloor\frac{n^{2}}{4}\Rfloor
\end{equation}
Therefore, equalities in \eqref{EQEQ} holds, it implies that 

\begin{claim}\label{OneToOne}%
For every missing edge $e$ in $A$ or $B \cup C$, there exists precisely one quasi-edge of $e$ in $[A, B]$; conversely, for every edge in $[A, B]$, it is the quasi-edge of a missing edge in $A$ or $B \cup C$. Moreover, $|B \cup C| = |A| + 1$ or $|A| + 2$.
\end{claim}

\begin{claim}\label{X}%
If $u_{1}, u_{2} \in A$ and $v_{1}, v_{2} \in B$, $\{u_{1}v_{1}, u_{2}v_{2}\} \subseteq E(G^{c})$ and $\{u_{1}v_{2}, u_{2}v_{1}\} \subseteq E(G)$, then $\{u_{1}u_{2}, v_{1}v_{2}\} \subseteq E(G)$.
\end{claim}
\begin{proof}%
If $u_{1}u_{2} \notin E(G)$, then both $u_{1}v_{2}$ and $u_{2}v_{1}$ are quasi-edge of $u_{1}u_{2}$, a contradiction. Similarly, we can prove that $v_{1}v_{2} \in E(G)$.
\end{proof}

\begin{claim}\label{Nei}%
If $u_{1}u_{2}$ is a missing edge in $A$ and $\deg_{B}(u_{1}) \geq \deg_{B}(u_{2})$, then $N_{B}(u_{1}) = N_{B}(u_{2}) \cup \{y\}$, where $y$ is the end (in $B$) of the quasi-edge of $u_{1}u_{2}$. Similarly, if $v_{1}v_{2}$ is a missing edge in $B$ and $\deg_{A}(v_{1}) \geq \deg_{A}(v_{2})$, then $N_{A}(v_{1}) = N_{A}(v_{2}) \cup \{x\}$, where $x$ is the end (in $A$) of the quasi-edge of $v_{1}v_{2}$. Consequently, the missing edges in $A$ (resp. in $B$) form a bipartite graph on $A$ (resp. on $B$).
\end{claim}
\begin{proof}%
Let $u_{1}u_{2}$ be a missing edge in $A$. Suppose that $N_{B}(u_{1}) \nsubseteqq N_{B}(u_{2})$ and $N_{B}(u_{2}) \nsubseteqq N_{B}(u_{1})$. Choose a vertex $v_{1} \in N_{B}(u_{2}) \setminus N_{B}(u_{1})$ and a vertex $v_{2}$ in $N_{B}(u_{1}) \setminus N_{B}(u_{2})$, then $\{u_{1}v_{1}, u_{2}v_{2}\} \subseteq E(G^{c})$ and $\{u_{1}v_{2}, u_{2}v_{1}\} \subseteq E(G)$, by Claim \ref{X}, we have $u_{1}u_{2} \in E(G)$, a contradiction. Hence $N_{B}(u_{1}) \supseteq N_{B}(u_{2})$. If $|N_{B}(u_{1}) \setminus N_{B}(u_{2})| \geq 2$, then there are at least two quasi-edge of the missing edge of $u_{1}u_{2}$, a contradiction. Therefore, $N_{B}(u_{1}) = N_{B}(u_{2}) \cup \{y\}$. Similarly, we can prove that $N_{A}(v_{1}) = N_{A}(v_{2}) \cup \{x\}$, if $v_{1}v_{2}$ is a missing edge in $B$.

In the graph formed by the missing edges in $A$, one part $X$ is the vertices of degree odd in $B$, and the other part $Y$ is the vertices of degree even in $B$. For any missing edge $uv$, $\deg_{B}(u)$ and $\deg_{B}(v)$ differ by exactly one, so one is odd and the other is even, and hence $uv$ has one end in $X$ and the other in $Y$, then the graph is bipartite. Similarly, the graph formed by the missing edges in $B$ is a bipartite graph.
\end{proof}
\begin{claim}\label{NOB}%
There exists no vertex in $B$ which dominates $A$.
\end{claim}
\begin{proof}%
Suppose that there exists a vertex $v$ in $B$ which dominates $A$. Let $A = \{u_{1}, u_{2}, \dots, u_{k}\}$ and $\un(u_{i}v) = v_{i}$ for $i \in \{1, 2, \dots, k\}$. There are $|A|$ edges in $[A, \{v\}]$, then there are $|A|$ missing edges which are incident with $v$ in $B \cup C$. Consider $G + u_{0}v$, since $\{u_{0}, v\}$ does not dominate $G$, there exists a vertex $z$ such that $u_{0}z \mapsto v$ or $vz \mapsto u_{0}$. If $u_{0}z \mapsto v$, then $u_{0}z \in E(G)$ and $z \in A$, but $\{u_{0}, z\}$ does not dominate $C$, a contradiction. We may assume that $vz \mapsto u_{0}$. Since $zu_{0} \notin E(G)$, $z \in B \cup C$ and $vz \in E(G)$. Then $B \cup C = \{v_{1}, v_{2}, \dots, v_{k}, v, z\}$ by Claim~\ref{OneToOne} and $z$ dominates $\{v_{1}, v_{2}, \dots, v_{k}\}$. If $k=1$, then $G$ is a path of length four, it is not a $3$-$\gamma_{t}$-edge critical graph, a contradiction. So $k \geq 2$. Since the quasi-edge of $v_{i}v$ is $u_{i}v$, $u_{i}$ dominates $\{v_{1}, v_{2}, \dots, v_{k}\}\setminus \{v_{i}\}$, where $i \in \{1, 2, \dots, k\}$, then $\{v_{1}, v_{2}, \dots, v_{k}\} \subseteq B$. Moreover, $B = \{v_{1}, v_{2}, \dots, v_{k}, v\}$ and $C = \{z\} = \{v_{0}\}$. For $1 \leq i < j \leq k$, we have $\{u_{i}v_{i}, u_{j}v_{j}\} \subseteq E(G^{c})$ and $\{u_{i}v_{j}, u_{j}v_{i}\} \subseteq E(G)$, by Claim \ref{X}, $\{u_{i}u_{j}, v_{i}v_{j}\} \subseteq E(G)$, hence $A$ is a clique and $B\setminus \{v\}$ is also a clique. There are $k$ missing edges in $B \cup C$, but $e_{G}(A, B) = k(k+1)-k> k\, (k > 1)$, a contradiction. 
\end{proof}

\begin{claim}\label{BCFull}%
If $|C| \geq 2$, then $[B, C]$ is full.
\end{claim}
\begin{proof}%
Let $xy$ be a missing edge in $[B, C]$, where $x \in B$ and $y \in C$. Consider $G + u_{0}y$. Since $\{u_{0}, y\}$ does not dominate $x$, there exists a vertex $z$ such that $u_{0}z \mapsto y$ or $zy \mapsto u_{0}$. If $u_{0}z \mapsto y$, then $u_{0}z \in E(G)$ and $z \in A$, but $\{u_{0}, z\}$ does not dominate $C\setminus \{y\}$, a contradiction. We may assume that $zy \mapsto u_{0}$. Since $zy \in E(G)$, $z \in B$ and $z$ dominates $A$, which contradicts Claim \ref{NOB}.
\end{proof}
\begin{claim}\label{Aclique}%
If $|C| \geq 2$, then $A$ is a clique.
\end{claim}
\begin{proof}%
Suppose to the contrary that $xy$ is a missing edge in $A$. Consider $G + xv_{0}$. Neither $x$ nor $v_{0}$ dominate $y$, then there exists a vertex $z$ such that $xz \mapsto v_{0}$ or $zv_{0} \mapsto x$. If $zv_{0} \mapsto x$, then $zv_{0} \in E(G)$ and $\{z, v_{0}\}$ does not dominate $u_{0}$, a contradiction. We may assume that $xz \mapsto v_{0}$, then $zv_{0} \notin E(G)$, by Claim \ref{CClique} and \ref{BCFull}, $z \in \{u_{0}\} \cup A$, but $\{x, z\}$ does not dominate $C\setminus \{v_{0}\}$, a contradiction.
\end{proof}

\begin{claim}%
$|C| = 1$
\end{claim}
\begin{proof}%
Suppose that $|C| \geq 2$. If $B$ is a clique, then $B \cup C$ and $A$ are all cliques by Claim \ref{CClique}, \ref{BCFull} and \ref{Aclique}, consequently, $e_{G}(A, B)= 0$ by Claim~\ref{OneToOne} and $G$ is disconnected, a contradiction. We may assume that $B$ is not a clique. Let $xy$ be a missing edge in $B$. Consider $G + u_{0}x$. Neither $u_{0}$ nor $x$ dominates $y$, then there exists a vertex $z$ such that $u_{0}z \mapsto x$ or $zx \mapsto u_{0}$. If $u_{0}z \mapsto x$, then $u_{0}z \in E(G)$ and $z \in A$, but $\{u_{0}, z\}$ does not dominate $C$, a contradiction. We may assume that $zx \mapsto u_{0}$. Since $u_{0}z \notin E(G)$, $z \in B\cup C$, indeed $z \in B$; otherwise, $z \in C$ and $x$ dominates $A$, which contradicts Claim \ref{NOB}. Since $\{z, x\} \subseteq B$ dominates $A$, $e_{G}(A, \{z, x\}) \geq |A|$. By Claim \ref{Aclique}, $A \cup \{u_{0}\}$ is a clique, for any edge $e$ in $[A, \{z, x\}]$, $\un(e) \in B \cup C$, and thus $\un(e) \in B \setminus \{z, x\}$ since $[B, C]$ is full and $zx \in E(G)$. But $|B \setminus \{z, x\}| < |A|$, therefore, there exists $\{e, e'\} \in [A, \{z, x\}]$ such that $\un(e) = \un(e')=w \in B\setminus \{x, z\}$. By Claim \ref{OneToOne}, $e$ and $e'$ has no common end in $B$, hence $\{xw, zw\} \in E(G^{c})$, which contradicts the fact that $\{x, z\}$ totally dominates $G - u_{0}$.
\end{proof}

\begin{claim}\label{NOA}
No vertex in $A$ dominates $B$.
\end{claim}
\begin{proof}%
Suppose that $u \in A$ dominates $B$. Hence, for every edge $e \in [u, N_{B}(v_{0})]$, $\un(e) \in A$, and for different edge $e, e' \in [u, N_{B}(v_{0})]$, $\un(e) \neq \un(e')$. Therefore, $|A| \geq |\{u\} \cup \{\un(e) \mid e \in [u, N_{B}(v_{0})]\}| \geq 1+ \deg_{G}(v_{0}) \geq 1+\deg_{G}(u_{0})$, a contradiction. 
\end{proof}
\begin{claim}
$N_{B}(v_{0}) = B$.
\end{claim}
\begin{proof}
Suppose that $N_{B}(v_{0}) \varsubsetneq B$. Consider $G + u_{0}v_{0}$. Since $\{u_{0}, v_{0}\}$ does not dominate $G$, there exists a vertex $z$ such that $u_{0}z \mapsto v_{0}$ or $zv_{0} \mapsto u_{0}$. If $u_{0}z \mapsto v_{0}$, then $u_{0}z \in E(G)$, $z \in A$ and $z$ dominates $B$, which contradicts Claim \ref{NOA}. If $zv_{0} \mapsto u_{0}$, then $zv_{0} \in E(G)$, $z \in B$ and $z$ dominates $A$, which contradicts Claim \ref{NOB}.
\end{proof}
\begin{claim}\label{HalfD}%
(a) There exists a vertex $w$ in $A$ such that $w$ does not dominates $A$ and $\deg_{B}(w) > \frac{|B|}{2}$. Otherwise, (b) $A$ is a clique and there exists a vertex $w$ in $B$ such that $w$ does not dominates $B$ and $\deg_{A}(w) > \frac{|A|}{2}$.
\end{claim}
\begin{proof}%
Suppose that $A$ is not a clique, let $u_{1}u_{2}$ be a missing edge in $A$. By Claim \ref{Nei}, we may assume that $\deg_{B}(u_{1}) = \deg_{B}(u_{2}) + 1$. If $\deg_{B}(u_{1}) > \frac{|B|}{2}$, then we are done by taking $w = u_{1}$. Then we may assume that $\deg_{B}(u_{1}) \leq \frac{|B|}{2}$, \ie $\deg_{B}(u_{1}) \leq \Lfloor\frac{|B|}{2}\Rfloor$, and thus $\deg_{B}(u_{2}) \leq \Lfloor\frac{|B|}{2}\Rfloor -1$.

Consider $G + u_{2}v_{0}$. Neither $u_{2}$ nor $v_{0}$ dominate $u_{1}$, then there exists a vertex $z$ such that $zv_{0} \mapsto u_{2}$ or $u_{2}z \mapsto v_{0}$. If $zv_{0} \mapsto u_{2}$, then $zv_{0} \in E(G)$ and $z \in B$, but $\{z, v_{0}\}$ does not dominate $u_{0}$, a contradiction. So we have $u_{2}z \mapsto v_{0}$, then $zv_{0} \notin E(G)$ and $z \in A$ by Claim \ref{NOA}. Since $\deg_{B}(u_{2}) \leq \Lfloor\frac{|B|}{2}\Rfloor -1$, $\deg_{B}(z) \geq \Lceil \frac{|B|}{2} \Rceil + 1 > \frac{|B|}{2}$. If $z$ does not dominate $A$, then we are done by taking $w=z$. Hence, we may assume that $z$ dominates $A$. Let $\{y\} = N_{B}(u_{1}) \setminus N_{B}(u_{2})$. Since $yu_{2} \notin E(G)$, $yz \in E(G)$. Let $x = \un(yz)$. Since $\{y, z\}$ dominates $\{u_{0}, v_{0}\} \cup A$, we have $x \in B$, then $\{yu_{2}, xz, xy\} \subseteq E(G^{c})$ and $\{yz, xu_{2}\} \subseteq E(G)$ (Since $\{u_{2}, z\}$ dominates $G - v_{0}$, $xu_{2} \in E(G)$), which contradicts Claim \ref{X}.


Then we may assume that $A$ is a clique. Similarly, we can prove that there exists a vertex $w$ in $B$ such that $w$ does not dominate $B$ and $\deg_{A}(w) > \frac{|A|}{2}$.
\end{proof}


Let $\{U, W\} = \{A, B\}$. By Claim~\ref{HalfD}, we may assume that there exists a vertex $w$ in $W$ such that $w$ does not dominates $W$ and $\deg_{U}(w) > \frac{|U|}{2}$. Without loss of generality, among all such vertices in $W$, we may assume that $w$ is chosen such that $\deg_{U}(w)$ is maximum. 

\begin{claim}%
For every edge $e$ in $[\{w\}, N_{U}(w)]$, we have $\un(e) \in W$.
\end{claim} 
\begin{proof}%
Otherwise, assume that $wv \in [\{w\}, N_{U}(w)]$ and $\un(wv) = y \in U$. Let $wx$ be a missing edge in $W$. Since $wv \mapsto y$, $xv \in E(G)$. By Claim~\ref{Nei}, we have $N_{W}(v) = N_{W}(y) \cup \{w\}$ and hence $xy \in E(G)$. Now, we have $xy \in E(G)$ and $wy \notin E(G)$, \ie $y \in N_{U}(x) \setminus N_{U}(w)$, by Claim~\ref{Nei} again, we have $N_{U}(x) = N_{U}(w) \cup \{y\}$, which contradicts the fact that $\deg_{U}(w)$ is maximum among all the vertices in $W$ satisfying Claim~\ref{HalfD}.
\end{proof}

Let $N_{U}(w) = U_{1} = \{v_{1}, v_{2}, \dots, v_{\ell}\}$ and let $w_{i} = \un(wv_{i})$, where $i = 1, 2, \dots, \ell$. Then $w_{i} \neq w_{j}$ for $1 \leq i < j \leq \ell$; otherwise, both $wv_{i}$ and $wv_{j}$ are quasi-edges of $ww_{i}$, which contradicts Claim~\ref{OneToOne}. Let $W_{1} = \{w_{1}, w_{2}, \dots, w_{\ell}\}$. By Claim~\ref{Nei}, we have $N_{U}(w) = N_{U}(w_{i}) \cup \{v_{i}\}$, and by Claim~\ref{X}, $W_{1}$ and $U_{1}$ are all cliques. Moreover, every vertex $v_{i}$ dominates $U$ for $i = 1, 2, \dots, \ell$. Hence, for every edge $e$ in $[\{w_{1}\}, N_{U}(w_{1})]$, we have $\un(e) \in W\setminus (W_{1} \cup \{w\})$. Therefore, 
\begin{equation}\label{Com}%
|W| \geq |W_{1} \cup \{w\}| + |N_{U}(w_{1})| = \ell + 1 + (\ell-1) = 2\ell > |U|.
\end{equation}

If $W = A$ and $U = B$, then $|A| > |B|$, a contradiction. Then $W = B$ and $U = A$. From (\ref{Com}) and the fact that $|B| = |A|$ or $|A| + 1$, we conclude that $|B| = |A| + 1 = 2\ell$. For every edge in $[A, B]$, it is the quasi-edge of a missing edge in $B$ since $A$ is a clique by Claim~\ref{HalfD}. There are at least $\ell^{2} + \ell - 1$ edges in $[A, B]$, and there are at most $(2\ell)^{2}/4$ missing edges in $B$ by Claim \ref{Nei}. Therefore, $\ell = 1$, but it contradicts Claim~\ref{NOB}.
\end{proof}
\vskip 3mm \vspace{0.3cm} \noindent{\bf Acknowledgments.} This project was supported by NSFC (11026078).

\end{document}